\documentclass{amsart}

\usepackage{amsmath, amsthm, amscd, amsfonts, amssymb, graphicx, color,float,pgf,tikz}
\usepackage{amssymb,fontenc}
\usepackage{latexsym,wasysym,mathrsfs}
\usepackage{hyperref}
\usepackage{subcaption}
\usepackage{soul}

\newtheorem{theorem}{Theorem}[section]
\newtheorem{lemma}[theorem]{Lemma}
\newtheorem{proposition}[theorem]{Proposition}
\newtheorem{corollary}[theorem]{Corollary}

\theoremstyle{definition}
\newtheorem{definition}[theorem]{Definition}
\newtheorem{example}[theorem]{Example}

\title{A Combinatorial Problem in Cinema Seating}
\author[M. Mirzavaziri and D. Yaqubi ]{Madjid Mirzavaziri and Daniel Yaqubi$^{1*}$ }

\address{ $^{*}$ Department of Pure Mathematics, Faculty of Mathematical Sciences,  University of Torbat-e Jam, Torbat-e Jam, Iran.}
\email{Daniel\_yaqubi@yahoo.es, yaqubi@tjamcaas.ac.ir}

\address{ Department of Pure Mathematics, Faculty of Mathematical Sciences,  Ferdowsi University of Mashhad, P. O. Box 1159-91775, Mashhad, Iran.}
\email{mirzavaziri@gmail.com}



\begin{document}

\begin{abstract}
We address a problem concerning cinema audiences: ``A cinema has $n$ seats numbered from $1$ to $n$, and there are $n$ people with tickets numbered from $1$ to $n$. People enter the cinema in order. If someone has the ticket number $i$, they can choose seats whose numbers are multiples of $i$. They should exit the cinema if the permitted seats are occupied by previous audience members. In how many ways can they be seated under these conditions?" We give an algorithm to create the list of situations that meet these conditions. We also focus on finding the number of situations in two special cases: when exactly one seat is unoccupied whose total number is denoted by $\omega(n)$, and when all audiences $1, \ldots, n-1$ are seated, whose total number is denoted by $\psi(n)$. Giving the recursive formula $\psi(n)=1+\sum_{d|n, d\neq n}\psi(d)$ with the initial value $\psi(1)=1$, we provide an explicit formula for $\psi(p^\alpha q^\beta)$, where $p$ and $q$ are distinct prime numbers. Furthermore, we show that $\omega(n)=-n+\sum_{i=1}^n\psi(i)$.
\end{abstract}

\maketitle

\section{Introduction}

The \textit{Dedekind numbers}, named after the renowned German mathematician Julius Wilhelm Richard Dedekind, are a sequence of combinatorial numbers with a rich history and wide-ranging applications \cite{Dedekind1897, BermanKohler1976, Church1940, Kisielewicz1988, KleitmanMarkowsky1975, Wiedemann1991}. These numbers are intimately connected to various areas of mathematics, including number theory, combinatorics, and algebra. Originally introduced to enumerate the number of different modular lattices or order ideals in posets \cite{Dedekind1897}, Dedekind numbers have found their way into diverse fields. They appear in the study of distributive lattices \cite{BermanKohler1976}, graph theory \cite{Church1940}, and even in algebraic geometry \cite{Kisielewicz1988}. For example, they are associated with counting factorizations of algebraic integers into prime ideals in number fields, a fundamental concept in algebraic number theory \cite{KleitmanMarkowsky1975}. Additionally, Dedekind numbers have intriguing connections to partitions, Bell numbers, and the theory of partitions \cite{Kisielewicz1988, Korshunov1981, Zaguia1993}. Their presence across such a spectrum of mathematical domains underscores their significance and makes them a fascinating subject of exploration.

A Dedekind number, often denoted as $D_n$, is a combinatorial number that counts the number of \textit{non-increasing Boolean functions} on a finite poset $(\mathsf{P},\preceq)$ of size $n$. In other words, $D_n$ represents the number of non-increasing functions from the poset $\mathsf{P}$ to the two-element chain $\{0,1\}$. 

For a positive integer $n$, the set $\mathsf{P}_n=\{1,\ldots,n\}$ is a poset equipped with \textit{divisibility as order}, in the sense that $i\preceq j$ whenever $i|j$. Thus a boolean non-increasing function $\delta:\mathsf{P}_n\to\{0,1\}$ has the property that if $\delta(j)$ is non-zero, then $\delta(i)$ is non-zero for each divisor $i$ of $j$. 
This motivates us to consider all mappings $\Sigma:\mathsf{P}_n\to\mathsf{P}_n\cup\{0\}$ with the property that $\Sigma(j)\neq0$ implies $\Sigma(i)\neq0$ for each divisor $i$ of $j$. We add the extra property $i|\Sigma(i)$ for each $i\in\mathsf{P}_n$ to these kind of mappings and consider several enumeration and extreme problems about the so-called \textit{cinema mappings}. 

The above discussion shows that the idea of our approach is taken from the Dedekind Boolean non-increasing functions on the poset $(\mathsf{P}_n,|)$.

In this paper, we delve into the intriguing world of cinema mappings, which serve as a captivating framework for studying various seating arrangements in a cinema hall.  Additionally, we introduce a novel function, the \textit{total cinema function}, which captures the total count of cinema mappings for a given number of seats. We give an upper bound for the total cinema function which reveals a profound connection between cinema mappings and the multiplicative divisor function. We also give an algorithm to create all cinema mappings with $n$ seats. Our journey takes us through different scenarios, each shedding light on specific aspects of cinema mappings.

The central theme of this paper revolves around the classification and enumeration of cinema mappings based on two distinct scenarios: ``One Seat to Full" and ``Last Audience Is in the Waiting List." These scenarios pose intriguing questions and lead us to profound results in combinatorial theory.

In the first scenario, ``One Seat to Full," we explore cinema mappings where precisely one seat remains unoccupied. We define these mappings as \textit{one seat to full}, and our objective is to determine the number of such mappings for a given number of seats. We establish a connection between these mappings and a fundamental concept which we call it as \textit{conquester chains}, providing insights into the structural properties of cinema mappings.

Moving forward, the second scenario, ``Last Audience Is in the Waiting List," presents an equally captivating study. Here, we investigate cinema mappings where all but perhaps one seat are occupied, and the last audience remains in a waiting list. This intriguing scenario leads us to develop a recursive formula to calculate the number of these mappings, shedding light on the richness of this combinatorial problem. 

The $\psi$ function defined in the section ``Last Audience Is in the Waiting List" has been previously studied and is documented in the On-Line Encyclopedia of Integer Sequences (OEIS) as sequence A067824. This sequence provides further context and highlights the significance of our findings. We give an explicit formula for $\psi(n)$ in cases where $n$ has at most two distinct prime factors.

Throughout this paper, we employ various theorems and lemmas to support our findings and provide recursive and explicit formulas for counting special kinds of these cinema mappings.

\section{Preliminaries}
Let $n$ be a positive integer. Throughout the paper, we use the notation $[n]$ to denote the set $\{1,\ldots,n\}$. Recall that a \textit{poset} $(\mathsf{P},\preceq)$ is a set $\mathsf{P}$ together with a partial ordering $\preceq$ on it. As an example, we can consider the \textit{divisibility} as an order on $[n]$. In this sense, $([n],|)$ forms a poset.

If $(\mathsf{P},\preceq)$ and $(\mathsf{P}',\preceq')$ are two posets, then a function $\delta:\mathsf{P}\to \mathsf{P}'$ is called \textit{non-increasing} if $a\preceq b$ implies $\delta(b)\preceq'\delta(a)$ for each $a,b\in \mathsf{P}$. A \textit{Boolean function} on a poset $(\mathsf{P},\preceq)$ is a function whose range is in $\{0,1\}$. 

In this paper, we propose the following problem which is stated as follows: 

A cinema has $n$ seats numbered from $1$ to $n$, and there are $n$ people with tickets numbered from $1$ to $n$. People enter the cinema in order. If someone has the ticket number $i$, they can choose seats whose numbers are multiples of $i$. They should exit the cinema if the permitted seats are occupied by previous audience members. In how many ways can they be seated under these conditions?

We can formulate our problem as follows:

\begin{definition}
Let $n$ be a positive integer. An injective mapping $\sigma:D_\sigma\subseteq [n]\to R_\sigma\subseteq[n]$ is called a \textit{cinema mapping with $n$ seats} if it has the following two proprieties: $i|j$ and $j\in D_\sigma$ implies $i\in D_\sigma$, which we call it \textit{the divisor-first property}, and $i|\sigma(i)$ for each $i\in D_\sigma$, which is mentioned as \textit{the seat-multiple-ticket property}. The domain $D_\sigma$ is \textit{the list of seated audiences} and the range $R_\sigma$ is \textit{the list of occupied seats}. When $\sigma(i)=j$ we say that $i$ \textit{is seated in seat number} $j$. 
\end{definition}

Throughout the paper, when we write $\sigma:[n]\to[n]$ is a cinema mapping, it means that $\sigma:D_\sigma\subseteq [n]\to R_\sigma\subseteq[n].$

Our motivation for considering cinema mappings and related problems stems from their close connection to non-increasing Boolean functions. Consider a mapping $\Sigma:D_\Sigma\to R_\Sigma$, where $D_\Sigma$ is a subset of $[n]$ and $R_\Sigma$ is a subset of $[n]\cup\{0\}$, such that for each $i$ in $D_\Sigma$, $i$ divides $\Sigma(i)$, and $\Sigma$ is injective when restricted to $\Sigma^{-1}([n])$. We define $\delta:D_\delta=[n]\to\{0,1\}$ such that $\delta(i)=0$ if $\Sigma(i)=0$, and $\delta(i)=1$ if $\Sigma(i)\neq0$. Now, the mapping $\sigma:D_\sigma=\Sigma^{-1}([n])\subseteq[n]\to[n]$ defined by $\sigma=\Sigma$ is a cinema mapping if and only if $\delta$ is a non-increasing Boolean function. This connection is fundamental to understanding how audiences can be seated under specific conditions. In this context, audiences who cannot find seats correspond to $i$ with $\Sigma(i)=0$. Additionally, the non-increasing property of the Boolean function $\delta$ signifies that people choose their seats in order from $1$ to $n$—if $i$ cannot find a seat, then multiples of $i$ also cannot.

These considerations show the relationship between the Dedekind's numbers and the problem of enumerating the cinema mappings. 

To illustrate the challenges associated with cinema mappings, let's begin by examining a case with a small value of $n$. This will provide us with a concrete example that highlights the intricacies and nuances of the problem. As we delve deeper into the analysis, we will develop algorithms, explore upper bounds, and tackle specific scenarios, shedding light on the rich world of cinema mappings and their applications in combinatorial theory.

\begin{example}\label{n=5,6}
We have compiled a complete list of cinema mappings for $5$ and $6$ seats, which can be found in Tables $1$ and $2$, respectively. In these tables, the first row represents the row-seat arrangement, and any unoccupied seat is denoted by $\boxdot$. Consider Table $2$. Here we have some enumeration and extreme problems:
\begin{itemize}
\item How many cinema mappings exist in total? There are $21$.
\item How many cinema mappings exist with the property that audience $i$ is seated in their seat for a fixed $i$? The counts are $5, 8, 11, 9, 16, 6$ for $i=1, 2, 3, 4, 5, 6$, respectively.
\item How many cinema mappings exist with the property that audience $i$ is seated in the seat $j$ for fixed $i$ and $j$? For example, there are $5$ situations in which audience $2$ is seated in seat number $6$.
\item Which cinema mappings have exactly $1$ unoccupied seat? The cinema mappings $\sigma_2, \sigma_3, \sigma_5, \sigma_6, \sigma_8, \sigma_9, \sigma_{12}, \sigma_{15},$ and $\sigma_{19}$ fall into this category.
\item Which cinema mappings have audiences $1, \ldots, n-1$ seated? The cinema mappings $\sigma_1, \sigma_2, \sigma_5, \sigma_8, \sigma_9,$ and $\sigma_{19}$ satisfy this condition.
\item Which cinema mappings have the least number of seats occupied? Only $\sigma_{18}$.
\end{itemize}

\begin{table}[H]
\centering

        \begin{tabular}{ r | c c c c c}
         & 1 & 2 & 3 & 4 & 5 \\
        \hline
            $\rho_1$ & 1 & 2 & 3 & 4 & 5 \\
            $\rho_2$ & 1 & $\boxdot$ & 3 & 2 & 5 \\
            $\rho_3$ & $\boxdot$ & 1 & 3 & 2 & 5  \\
            $\rho_4$ & $\boxdot$ & 2 & 1 & 4 & 5 \\
            $\rho_5$ & $\boxdot$ & $\boxdot$ & 1 & 2 & 5 \\
            $\rho_6$ & $\boxdot$ & 2 & 3 & 1 & 5  \\
            $\rho_7$ & $\boxdot$ & 2 & 3 & 4 & 1  \\
            $\rho_8$ & $\boxdot$ & $\boxdot$ & 3 & 2 & 1
         
         \end{tabular}
\caption{List of All Cinema Mappings with $5$ Seats}
\end{table}

\begin{table}[H]
\begin{subtable}[l]{0.4\textwidth}
\centering
        \begin{tabular}{ r | c c c c c c}
         & 1 & 2 & 3 & 4 & 5 & 6 \\
        \hline
            $\sigma_1$ & 1 & 2 & 3 & 4 & 5 & 6 \\
            $\sigma_2$ & 1 & 2 & $\boxdot$ & 4 & 5 & 3 \\
            $\sigma_3$ & 1 & $\boxdot$ & 3 & 2 & 5 & 6 \\
            $\sigma_4$ & 1 & $\boxdot$ & $\boxdot$ & 2 & 5 & 3 \\
            $\sigma_5$ & 1 & $\boxdot$ & 3 & 4 & 5 & 2 \\
            $\sigma_6$ & $\boxdot$ & 1 & 3 & 2 & 5 & 6 \\
            $\sigma_7$ & $\boxdot$ & 1 & $\boxdot$ & 2 & 5 & 3 \\
            $\sigma_8$ & $\boxdot$ & 1 & 3 & 4 & 5 & 2 \\
            $\sigma_9$ & $\boxdot$ & 2 & 1 & 4 & 5 & 3 \\
            $\sigma_{10}$ & $\boxdot$ & $\boxdot$ & 1 & 2 & 5 & 3 \\
            $\sigma_{11}$ & $\boxdot$ & $\boxdot$ & 1 & 4 & 5 & 2 
         \end{tabular}
         \subcaption*{$\sigma(1)=1,2,$ or $3$}
\end{subtable}
\begin{subtable}[r]{0.4\textwidth}
\centering
		\begin{tabular}{ r | c c c c c c}
         & 1 & 2 & 3 & 4 & 5 & 6 \\
        \hline
            $\sigma_{12}$ & $\boxdot$ & 2 & 3 & 1 & 5 & 6 \\
            $\sigma_{13}$ & $\boxdot$ & 2 & $\boxdot$ & 1 & 5 & 3 \\
            $\sigma_{14}$ & $\boxdot$ & $\boxdot$ & 3 & 1 & 5 & 2 \\
            $\sigma_{15}$ & $\boxdot$ & 2 & 3 & 4 & 1 & 6 \\
            $\sigma_{16}$ & $\boxdot$ & 2 & $\boxdot$ & 4 & 1 & 3 \\
            $\sigma_{17}$ & $\boxdot$ & $\boxdot$ & 3 & 2 & 1 & 6 \\
            $\sigma_{18}$ & $\boxdot$ & $\boxdot$ & $\boxdot$ & 2 & 1 & 3 \\
            $\sigma_{19}$ & $\boxdot$ & $\boxdot$ & 3 & 4 & 1 & 2 \\
            $\sigma_{20}$ & $\boxdot$ & 2 & 3 & 4 & 5 & 1 \\
            $\sigma_{21}$ & $\boxdot$ & $\boxdot$ & 3 & 2 & 5 & 1 \\
            &&&&&&
        \end{tabular}
        \subcaption*{$\sigma(1)=4,5,$ or $6$}
\end{subtable}
\caption{List of All Cinema Mappings with $6$ Seats}
\end{table}
\end{example}

In the subsequent sections of this paper, we will delve into a series of intriguing problems related to cinema mappings. Among these problems, our primary objectives include developing an algorithm for systematically generating the complete list of cinema mappings with $n$ seats. We will also establish an upper bound for the total number of possible cinema mappings. 

\section{Total Cinema Function}

In the forthcoming sections, we will delve deeper into the world of cinema mappings and explore various intriguing questions related to seating arrangements in a cinema hall. These questions form the basis for our algorithm to create the list of all cinema mappings with $n$ seats. To achieve this, we will employ a series of lemmas and definitions that facilitate our approach. In particular, Lemma \ref{extend} will play a pivotal role in extending existing cinema mappings to accommodate an additional seat, offering insights into the structural properties of these mappings.

\begin{lemma}\label{extend}
Let $n$ be a positive integer and let $\sigma:[n-1]\to[n-1]$ be a cinema mapping with $n-1$ seats. Then the following hold:
\begin{enumerate}
\item[i.] Suppose that $d$ is the least divisor of $n$ with $d\notin D_\sigma$. Extend $\sigma$ to the mapping $\sigma_d:D_{\sigma_d}=D_\sigma\cup\{d\}\subseteq[n]\to R_\sigma\cup\{n\}\subseteq[n]$ by $\sigma_d(d)=n$. Then $\sigma_d$ is a cinema mapping with $n$ seats.
\item[ii.] Let $d\in D_\sigma$, where $d|n$. Suppose also that there is a multiple $d'<n$ of $d$ which is the least multiple of $d$, with the property that $d'|\sigma(d)$ and $d'$ does not belong to $D_\sigma$. Change $\sigma$ to the mapping $\sigma_{d,d'}:D_{\sigma_{d,d'}}=D_\sigma\cup\{d'\}\subseteq[n]\to R_\sigma\cup\{n\}\subseteq[n]$ by $\sigma_{d,d'}(d)=n, \sigma_{d,d'}(d')=\sigma(d),$ and $\sigma_{d,d'}(i)=\sigma(i)$ for $i\in D_\sigma$ with $i\neq d$. Then $\sigma_{d,d'}$ is a cinema mapping with $n$ seats. 
\item[iii.] Let $d\in D_\sigma$, where $d|n$. Suppose also that for each multiple $d'<n$ of $d$ which is not in $D_\sigma$ we have $d'\not| \sigma(d)$. Change $\sigma$ to the mapping $\sigma_d:D_{\sigma_d}=D_\sigma\subseteq[n]\to (R_\sigma\setminus\{\sigma(d)\})\cup\{n\}$ by $\sigma_d(d)=n$ and $\sigma_d(i)=\sigma(i)$ for $i\in D_\sigma$ with $i\neq d$. Then $\sigma_d$ is a cinema mapping with $n$ seats. 
\end{enumerate}
\end{lemma}

\begin{proof}
We should only show that the divisor-first property holds, which is obvious by the assumption of each case. 
\end{proof}

Note that case (i) of Lemma \ref{extend} always occurs since we know that $n$ is a divisor of $n$ that is not in $D_\sigma$.

\begin{definition}
Let $n$ be a positive integer and let $\sigma:[n-1]\to[n-1]$ be a cinema mapping with $n-1$ seats. Then according to the cases (i), (ii), and (iii) of Lemma \ref{extend} we have the following situations:
\begin{enumerate}
\item[i.] The cinema mapping derived by \textit{adding the audience $d$} in case (i) is denoted by ${\rm Add}_d(\sigma)$. 
\item[ii.] The cinema mapping derived by \textit{forwarding the audience $d$ replaced by $d'$} in case (ii) is denoted by ${\rm Forw}_{d,d'}(\sigma)$. 
\item[iii.] The cinema mapping derived by \textit{voiding the seat $d$} in case (iii) is denoted by ${\rm Void}_d(\sigma)$. 
\end{enumerate}
\end{definition}

\begin{lemma}\label{restrict}
Let $n$ be a positive integer and let $\sigma:[n]\to[n]$ be a cinema mapping with $n$ seats. Then the following hold:
\begin{enumerate}
\item[i.] Suppose that $\sigma(d)=n$ and let no multiple of $d$ be in $D_\sigma$. Restrict $\sigma$ to the mapping $\sigma_d:D_{\sigma_d}=D_\sigma\setminus\{d\}\subseteq[n-1]\to R_\sigma\setminus\{n\}\subseteq[n-1]$ by $\sigma_d=\sigma$. Then $\sigma_d$ is a cinema mapping with $n-1$ seats.
\item[ii.] Let $\sigma(d) = n$ for some $d\neq n$, and suppose that all members of $[n-1]$ which are multiples of $d$ are included in $R_\sigma$. Let $d'$ be the greatest multiple of $d$ in $R_\sigma$. Change $\sigma$ to the mapping $\sigma_{d,d'}:D_{\sigma_{d,d'}}=D_\sigma\setminus\{\sigma^{-1}(d')\} \subseteq[n-1]\to R_\sigma\setminus\{n\}\subseteq[n-1]$ by $\sigma_{d,d'}(d)=d'$ and $\sigma_{d,d'}(i)=\sigma(i)$ for $i\in D_\sigma$ with $i\neq d$. Then $\sigma_{d,d'}$ is a cinema mapping with $n-1$ seats. 
\item[iii.] Let $\sigma(d)=n$ and suppose also that there is a multiple $d'<n$ of $d$ which is not in $R_\sigma$. Change $\sigma$ to the mapping $\sigma_{d,d'}:D_{\sigma_{d,d'}}=D_\sigma \subseteq[n-1]\to (R_\sigma\cup\{d'\})\setminus\{n\}\subseteq[n-1]$ by $\sigma_{d,d'}(d)=d'$ and $\sigma_{d,d'}(i)=\sigma(i)$ for $i\in D_\sigma$ with $i\neq d'$. Then $\sigma_d$ is a cinema mapping with $n-1$ seats. 
\end{enumerate}
\end{lemma}

\begin{proof}
We should only show that the divisor-first property holds, which is obvious by the assumption of each case. 
\end{proof}

Note that case (i) of Lemma \ref{restrict} occurs at least when $\sigma(n)=n$. This is the only case with the property that $n\in D_\sigma$.

\begin{definition}
Let $n$ be a positive integer and let $\sigma:[n]\to[n]$ be a cinema mapping with $n$ seats. Then according to the cases (i), (ii), and (iii) of Lemma \ref{restrict} we have the following situations:
\begin{enumerate}
\item[i.] The cinema mapping derived by \textit{deporting the audience $d$} in case (i) is denoted by ${\rm Dep}_d(\sigma)$. 
\item[ii.] The cinema mapping derived by \textit{replacing the audience $d'$ by $d$} in case (ii) is denoted by ${\rm Rep}_{d,d'}(\sigma)$. 
\item[iii.] The cinema mapping derived by \textit{backwarding the audience $d$ in seat $d'$} in case (iii) is denoted by ${\rm Back}_{d,d'}(\sigma)$. 
\end{enumerate}
\end{definition}

\begin{definition}
Let $n$ be a positive integer and let $\rho,\sigma:[n-1]\to[n-1]$ be two cinema mappings. We say that $\rho$ and $\sigma$ are \textit{apparently equal}, denoted by $\rho\simeq\sigma$ if one of the followings holds:
\begin{enumerate}
\item[i.] There is a $k\in D_{\rho}=D_{\sigma}$ such that $\rho(i)=\sigma(i)$ for all $i\in[n-1]$ with $i\neq k$, i.e., at most one person changes their seat. In this case, to emphasize the role of $k$, we use the notation $\rho\simeq_k\sigma$. 
\item[ii.] There are $k,\ell\in[n-1]$ such that $k|\ell, D_{\sigma}=D_{\rho}\cup\{\ell\}, \rho(k)=\sigma(\ell), \sigma(k)\notin R_{\rho}$, and $\rho(i)=\sigma(i)$ for all $i\in[n-1]$ with $i\neq k,\ell$, i.e., at most one person takes the seat of another person, whose ticket is a multiple of the first one, and the second person leaves the cinema. In this case, to emphasize the role of $k$ and $\ell$, we use the notation $\rho\simeq_{k,\ell}\sigma$. 
\end{enumerate} 
\end{definition}

\begin{example}\label{9pairs}
Returning to Example \ref{n=5,6}, we observe $9$ cases of apparently equal pairs in Table $1$:

\begin{center}
\begin{tabular}{ l l l }
$\rho_2\simeq_{2,4}\rho_1$, &$\rho_4\simeq_{1,3}\rho_1$, &$\rho_6\simeq_{1,4}\rho_1$, \\
$\rho_3\simeq_1\rho_2$, &$\rho_5\simeq_{1,3}\rho_3$, &$\rho_8\simeq_{1,5}\rho_3$, \\
$\rho_5\simeq_{2,4}\rho_4$, &$\rho_7\simeq_{1,5}\rho_1$, &$\rho_8\simeq_{2,4}\rho_7$.
\end{tabular}
\end{center}
\end{example}

\begin{proposition}
Let $n$ be a positive integer and let $\rho,\sigma:[n-1]\to[n-1]$ be two cinema mappings with $n-1$ seats. Suppose that $\bar{\rho},\bar{\sigma}:[n]\to[n]$ derived from $\rho,\sigma$ by the adding, forwarding, or voiding process. Then $\bar{\rho}=\bar{\sigma}$ if and only if $\rho\simeq\sigma$.
\end{proposition}

\begin{proof}
Let $\bar{\rho}=\bar{\sigma}$. We have six cases:
\begin{itemize}
\item Case I. $\bar{\rho}={\rm Add}_{c}(\rho)$ and $\bar{\sigma}={\rm Add}_{d}(\sigma)$.

We have $c=d$, since $c=\bar{\rho}^{-1}(n)=\bar{\sigma}^{-1}(n)=d$. Thus $\rho={\rm Dep}_{c}(\bar{\rho})={\rm Dep}_{d}(\bar{\sigma})=\sigma$.
\item Case II. $\bar{\rho}={\rm Add}_{c}(\rho)$ and $\bar{\sigma}={\rm Forw}_{d,d'}(\sigma)$.

We have $c=d$, since $c=\bar{\rho}^{-1}(n)=\bar{\sigma}^{-1}(n)=d$. We know that $c\notin D_{\rho}$ since $\bar{\rho}={\rm Add}_{c}(\rho)$. On the other hand, $d'\in D_{{\rm Forw}_{d,d'}(\sigma)}= D_{\bar{\sigma}}=D_{\bar{\rho}}= D_{{\rm Add}_{c}(\rho)}\subseteq D_{\rho}$. This contradicts to the first-divisor property, since $d'$ is a multiple of $d=c$. 
\item Case III. $\bar{\rho}={\rm Add}_{c}(\rho)$ and $\bar{\sigma}={\rm Void}_{d}(\sigma)$.

We have $c=d$, since $c=\bar{\rho}^{-1}(n)=\bar{\sigma}^{-1}(n)=d$. We know that $c=d\in D_{\sigma}=D_{{\rm Back}_{d,\sigma(d)}(\bar{\sigma})}= D_{\bar{\sigma}}=D_{\bar{\rho}}$, which contradicts to the fact that $c\notin D_{\bar{\rho}}$. 
\item Case IV. $\bar{\rho}={\rm Forw}_{c,c'}(\rho)$ and $\bar{\sigma}={\rm Forw}_{d,d'}(\sigma)$.

We have $c=d$, since $c=\bar{\rho}^{-1}(n)=\bar{\sigma}^{-1}(n)=d$. Moreover, $c'=\bar{\rho}^{-1}(\bar{\rho}(c'))=\bar{\sigma}^{-1}(\bar{\sigma}(d'))=d'$. Thus we have $\rho={\rm Rep}_{c,c'}(\bar{\rho})={\rm Rep}_{d,d'}(\bar{\sigma}) =\sigma$. 
\item Case V. $\bar{\rho}={\rm Forw}_{c,c'}(\rho)$ and $\bar{\sigma}={\rm Void}_{d}(\sigma)$.

We have $c=d$, since $c=\bar{\rho}^{-1}(n)=\bar{\sigma}^{-1}(n)=d$. In this case we see that $\rho\simeq_{c,c'}\sigma$. Note that $c|c', D_\sigma=D_{\bar{\sigma}}=D_{\bar{\rho}}=D_\rho\cup\{c'\},$
and $\rho(c)=\bar{\rho}(c')=\bar{\sigma}(c')=\sigma(c')$. Moreover, note that $\sigma(c)\notin R_\rho$, since $\sigma(c)=\sigma(d)\notin R_{{\rm Void}_d(\sigma)}=R_{\bar{\sigma}} =R_{\bar{\rho}}=R_\rho\cup\{n\}$.

\item Case VI. $\bar{\rho}={\rm Void}_{c}(\rho)$ and $\bar{\sigma}={\rm Void}_{d}(\sigma)$.

In this case we have $\rho\simeq_c\sigma$, since $c=d$.
\end{itemize}
\end{proof}

The preceding proposition provides a straightforward recursive algorithm for generating a comprehensive list of cinema mappings.

\begin{theorem}\label{alg}
Let $n$ be a positive integer and let $L_{n-1}$ be the list of all cinema mappings with $n-1$ seats. Then the following process creates the list $L_n$ containing all cinema mappings with $n$ seats:
\begin{itemize}
\item Select a cinema mapping $\sigma$ from $L_{n-1}$ and a divisor $d$ of $n$.
\item If $d$ is the least divisor of $n$ with $d\notin D_\sigma$ then add ${\rm Add}_d(\sigma)$ to $L_n$.
\item If $d\in D_\sigma$ and there is a multiple $d'<n$ of $d$ which is the least multiple of $d$, with the property that $d'|\sigma(d)$ and $d'$ does not belong to $D_\sigma$, and if $\sigma$ is not apparently equal to the previous cinema mappings in the list $L_{n-1}$, then add ${\rm Forw}_{d,d'}(\sigma)$ to $L_n$.
\item If $d\in D_\sigma$ and for each multiple $d'<n$ of $d$ which is not in $D_\sigma$ we have $d'\not| \sigma(d)$, and if $\sigma$ is not apparently equal to the previous cinema mappings in the list $L_{n-1}$, then add ${\rm Void}_d(\sigma)$ to $L_n$.
\item Repeat the process until there is no new cinema mappings in the list $L_{n-1}$.
\end{itemize}
\end{theorem}

Let us illustrate the algorithm mentioned in Theorem \ref{alg} with a simple example.

\begin{example}
Returning to Example \ref{n=5,6}, we will examine the process of creating Table $2$ using Table $1$. The cinema mappings with $6$ seats constructed via the cinema mappings with $5$ seats are demonstrated in Table $3$, where each row corresponds to a cinema mapping of Table $1$ for divisors $1, 2, 3$, and $6$ of $n=6$ from left to right.

\begin{table}[H]
\centering
\begin{tabular}{| l | l | l | l | }
\hline
$\sigma_{20}={\rm Void}_1(\rho_1)$ &
$\sigma_5={\rm Void}_2(\rho_1)$ &
$\sigma_2={\rm Void}_3(\rho_1)$ &
$\sigma_1={\rm Add}_6(\rho_1)$ \\
\hline
$\sigma_{21}={\rm Void}_1(\rho_2)$ &  
\st{$\sigma_5={\rm Forw}_{2,4}(\rho_2)$} & 
$\sigma_4={\rm Void}_3(\rho_2)$ &
$\sigma_3={\rm Add}_6(\rho_2)$ \\
\hline
\st{$\sigma_{21}={\rm Void}_1(\rho_3)$} & 
$\sigma_8={\rm Forw}_{2,4}(\rho_3)$ & 
$\sigma_7={\rm Void}_3(\rho_3)$ &
$\sigma_6={\rm Add}_6(\rho_3)$ \\
\hline
\st{$\sigma_{20}={\rm Forw}_{1,3}(\rho_4)$} & 
$\sigma_{11}={\rm Void}_2(\rho_4)$ &
$\sigma_9={\rm Add}_3(\rho_4)$ & 
 \\ 
\hline
\st{$\sigma_{21}={\rm Forw}_{1,3}(\rho_5)$} & 
\st{$\sigma_{11}={\rm Forw}_{2,4}(\rho_5)$} & 
$\sigma_{10}={\rm Add}_3(\rho_5)$ & 
  \\ 
\hline
\st{$\sigma_{20}={\rm Forw}_{1,4}(\rho_6)$} & 
$\sigma_{14}={\rm Void}_2(\rho_6)$ & 
$\sigma_{13}={\rm Void}_3(\rho_6)$ & 
$\sigma_{12}={\rm Add}_6(\rho_6)$ \\
\hline
\st{$\sigma_{20}={\rm Forw}_{1,5}(\rho_7)$} & 
$\sigma_{19}={\rm Void}_2(\rho_7)$ & 
$\sigma_{16}={\rm Void}_3(\rho_7)$ & 
$\sigma_{15}={\rm Add}_6(\rho_7)$ \\
\hline
\st{$\sigma_{21}={\rm Forw}_{1,5}(\rho_8)$} & 
\st{$\sigma_{19}={\rm Forw}_{2,4}(\rho_8)$} & 
$\sigma_{18}={\rm Void}_3(\rho_8)$ & 
$\sigma_{17}={\rm Add}_6(\rho_8)$ \\
\hline 
\end{tabular}
\caption{Cinema Mappings with $6$ Seats Derived from $5$ Seats}
\end{table}

Note that $9$ cells are stroked out because, according to our algorithm, we do not apply ${\rm Forw}$ and ${\rm Void}$ when a new cinema mapping in the list $L_{n-1}$ is apparently equal to one of the previous cinema mappings. These $9$ pairs are introduced in Example \ref{9pairs}. Furthermore, there are two empty cells related to ${\rm Add}_6(\rho_4)$ and ${\rm Add}_6(\rho_5)$. In these two cases, ${\rm Add}$ is not applicable since $6$ is not the least divisor of $n=6$ that is not in the domain of $\rho_4$ or the domain of $\rho_5$.

Thus, the number of cinema mappings with $6$ seats is $(4 \times 8) - 9 - 2 = 21$, where $8$ is the number of cinema mappings with $5$ seats, and $4$ is the number of divisors of $6$.
\end{example}

\begin{definition}
The function $\theta:\mathbb{N}\to\mathbb{N}$ defined by $\theta(n)=|L_n|$ is called the \textit{total cinema function}.
\end{definition}
Recall that the multiplicative divisor function $\tau:\mathbb{N}\to\mathbb{N}$ is defined by $\tau(n)=\sum_{d|n}1$. If $n=\prod_{i=1}^r p_i^{\alpha_i}$ is the prime factorization of $n$, then $\tau(n)=\prod_{i=1}^r(\alpha_i+1)$. The next corollary readily follows from Theorem \ref{alg}.
\begin{corollary}
Let $n$ be a positive integer. Then $\theta(n)\leqslant \prod_{i=1}^n\tau(i)$. 
\end{corollary}

\section{One Seat to Full}

In this section, we delve into a fascinating category of cinema mappings, which we refer to as \textit{one seat to full}. These mappings represent a unique arrangement of seats in a cinema hall, where only a single seat remains unoccupied. We will explore the properties and characteristics of these one seat to full cinema mappings and establish a compelling connection between them and a concept called \textit{conquester chains} within finite posets. By investigating this relationship, we aim to shed light on the intriguing combinatorial structures that underlie cinema seating arrangements.

\begin{definition}
Let $n$ be a positive integer. A cinema mapping $\sigma:[n]\to[n]$ is called \textit{one seat to full} if $|D_\sigma|=n-1$. The number of one seat to full cinema mappings with $n$ seats is denoted by $\omega(n)$.
\end{definition}

\begin{definition}
Let $(\mathsf{P},\preceq)$ be a finite poset and let $C=\{a_i\}_{i=1}^k$, with $c_1\prec \ldots\prec c_k$, be a chain in $\mathsf{P}$. We say that $C$ is a \textit{conquester chain} if $a_k$ is a maximal element of $\mathsf{P}$. 
\end{definition}

The term conquester is chosen to convey the idea that while these chains may not be maximal in themselves, they ascend towards the pinnacle of the poset, akin to conquering a mountaintop. This choice of terminology highlights the nature of these chains as they reach their zenith within the maximal elements of the poset.

\begin{proposition}
Let $n$ be a positive integer. Then there is a one-to-one correspondence between the set of all one seat to full cinema mappings with $n$ seats and the set of non-singleton conquester chains of $([n],|)$. 
\end{proposition}

\begin{proof}
Let $\sigma:[n]\to[n]$ be a one seat to full cinema mapping and let $[n]\setminus D_\sigma=\{d\}$, i.e., $d$ is the audience number which is outside the cinema. We know that $\sigma$ is a bijection between $D_\sigma$ and $R_\sigma$. Thus $|R_\sigma|$ is also equal to $n-1$. Put $c$ to be the number of unoccupied seat. Clearly, $c\neq d$, since otherwise all of the other audiences should seat on their seat and then $\sigma(c)$ should be $c$ itself. This means that $c$ is in the domain of $\sigma$. Now consider the sequence
\[c, \sigma(c),\sigma(\sigma(c)),\sigma(\sigma(\sigma(c))),\ldots\]
The sequence should be terminated on some point, in the sense that $\sigma^{k}(c)=d$ for some positive integer $k$. Thus $C=\{\sigma^i(c)\}_{i=1}^k$ is a chain in $([n],|)$.

We claim that $C$ is a conquester chain. Suppose that $d=\sigma^{k}(c)$ is not a maximal element of $([n],|)$. This means that there is a multiple $d'> d$ of $d$ in $[n]$. This contradicts to the divisor-first property, since $d\notin D_\sigma$ but $d'\in D_\sigma$. Note that $C$ is not a singleton.

On the other hand, if $C=\{j_i\}_{i=1}^k$ is a non-singleton conquester chain in $([n],|)$, then the mapping $\sigma:[n]\to[n]$ defined by $\sigma(j_i)=j_{i+1}$ for $1\leqslant i\leqslant k-1$ and $\sigma(i)=i$  for $i\neq j_1,\ldots,j_k$ is a one seat to full cinema mapping with $D_\sigma=[n]\setminus\{j_k\}$ and $R_\sigma=[n]\setminus\{j_1\}$. 
\end{proof}

\begin{definition}
Let $n$ be a positive integer, and define $\Delta_n = \{d: d|n\}$. The poset $(\Delta_n,|)$ is referred to as \textit{the divisor poset corresponding to} $n$. We denote the number of conquester chains in $(\Delta_n,|)$ by $\psi(n)$.
\end{definition}

Note that the only maximal element of $(\Delta_n,|)$, which is its maximum, is $n$ itself. In the following we use the notation $d\Vert n$ to denote $d$ is a proper divisor of $n$, in the sense that $d|n$ but $d\neq n$.

\begin{lemma}
Let $n$ be a positive integer. Then $\psi(n) =1+ \sum_{d\Vert n} \psi(d)$, with the initial value $\psi(1)=1$. Furthermore, $\psi(p^\alpha)=2^\alpha$ for any prime $p$ and any positive integer $\alpha$.
\end{lemma}

\begin{proof}
Let $C=\{d_i\}_{i=1}^k$ be a conquester chain in $(\Delta_n,|)$. If $k=1$ then $C$ is a singleton and if $k>1$ then $d_k=n$ and $d_{k-1}$ is a proper divisor of $n$. Hence $C'=\{d_i\}_{i=1}^{k-1}$ is a conquester chain in $(\Delta_{d_{k-1}},|)$. This proves that $\psi(n)= 1+ \sum_{d\Vert n} \psi(d)$ which is a multiplicative function. 

Remarkably, for every prime $p$ and each positive integer $\alpha$, we find that $\psi(p^\alpha) = 2^\alpha$. This is due to the fact that any conquester chain $\{c_i\}_{i=1}^k$ in $(\Delta_{p^\alpha},|)$ satisfies the property that $c_k=p^\alpha$, and $\{c_i\}_{i=1}^{k-1}$ is a subset of $\{1,p,\ldots,p^{\alpha-1}\}$. 
\end{proof}

\begin{theorem}
Let $n$ be a positive integer. Then the number of one seat to full cinema mappings is determined by $\omega(n)=-n+\sum_{i=1}^n\psi(i)$.
\end{theorem}

\begin{proof}
For a non-singleton conquester chain $C=\{j_i\}_{i=1}^k$ in $([n],|)$, we have $1\leqslant j_k\leqslant n-1$ or $j_k=n$. This means that $C$ is a non-singleton conquester chain in $([n-1],|)$ or a non-singleton conquester chain in $(\Delta_n,|)$, which gives $\omega(n)=\omega(n-1)+\psi(n)-1$. Solving the latter recursion for $\omega$, we arrive at the result.
\end{proof}

\section{Last Audience Is in the Waiting List}

In this section, we explore the fascinating world of cinema mappings with a distinct focus on an intriguing category \textit{last audience is in the waiting list}. We delve into the definitions, properties, and recursive structures that govern these cinema mappings, providing valuable insights into their combinatorial nature. Through our exploration, we uncover the recursive patterns that underlie the enumeration of these cinema mappings, allowing us to quantify their abundance.

\begin{definition}
Let $n$ be a positive integer. A cinema mapping $\sigma:[n]\to[n]$ with $n$ seats is called \textit{almost full} if $D_\sigma\supseteq[n-1]$. 
\end{definition}
In the following we prove that the number of almost full cinema mappings with $n$ seats is in fact equal to the number of conquester chains in $(\Delta_n,|)$ which is $\psi(n)$. 
\begin{theorem}
Let $n$ be a positive integer. Then the number of almost full cinema mappings satisfies the recursion $\psi(n)=1+\sum_{d\Vert n}\psi(d)$, with the initial value $\psi(1)=1$. 
\end{theorem}

\begin{proof}
There are two cases: $n\in D_\sigma$ or $n\notin D_\sigma$. The first case gives us the identity mapping and the second case occurs when $\sigma(d)=n$ for some $d\Vert n$. 

If the latter case occurs, then we divide the seats into two parts: Part I consisting of $1,\ldots,d$ and Part II consisting of $d+1,\ldots,n$.

Each audience member in Part II should sit in their assigned seat; otherwise, we would have to hold somebody other than $n$ behind the cinema door. Therefore, Part I should be almost full.

These considerations guarantee that $\psi(n)=1+\sum_{d\Vert n}\psi(d)$.
\end{proof}

\begin{lemma}\label{recurpsi}
Let $p$ be a prime number and let $m$ and $\alpha$ be two positive integers. Then $\psi(p^{\alpha+1} m)=2\psi(p^{\alpha} m)+\sum_{d\Vert m}\psi(p^{\alpha+1} d)$.
\end{lemma}

\begin{proof}
We have
\begin{align*}
\psi(p^{\alpha+1} m)&=1+\sum_{d\Vert p^{\alpha+1} m}\psi(d)\\
&=1+\sum_{d\Vert p^{\alpha} m}\psi(d)+ \psi(p^{\alpha} m) +\sum_{d\Vert m}\psi(p^{\alpha+1} d)\\
&=2\psi(p^{\alpha} m)+\sum_{d\Vert m}\psi(p^{\alpha+1} d).
\end{align*}

\end{proof}

\begin{lemma}\label{q1}
Let $p$ and $q$ be two distinct prime numbers and let $\alpha$ be a positive integer. Then $\psi(p^\alpha q)=2^\alpha(\alpha+2)$.
\end{lemma}

\begin{proof}
We use induction on $\alpha$. For $\alpha=1$ we have $\psi(pq)=1+\sum_{d\Vert pq}\psi(d)=1+\psi(1)+\psi(p)+\psi(q)=1+1+2+2=2^1(1+2)$. 

Let $\psi(p^\alpha q)=2^\alpha(\alpha+2)$. Then Lemma \ref{recurpsi} implies
\[\psi(p^{\alpha+1}q)=2\psi(p^\alpha q)+\sum_{d\Vert q}\psi(p^{\alpha+1}d)=2\cdot2^\alpha(\alpha+2)+\psi(p^{\alpha+1}) =2^{\alpha+1}(\alpha+3).\]
\end{proof}

In our concluding findings, as presented in Theorems \ref{pqfirst} and \ref{pqsecond}, we provide distinct methods for the explicit evaluation of $\psi(n)$ when $n$ has precisely two distinct prime factors. 

\begin{theorem}\label{pqfirst}
Let $p$ and $q$ be two distinct prime numbers and let $\alpha\geqslant \beta$ be two positive integers. Then $\psi(p^\alpha q^\beta)=2^\alpha \sum_{i=0}^\beta \binom{\beta}{i}\binom{\alpha+i}{i}$.
\end{theorem}

\begin{proof}
We use induction on $\alpha+\beta$. For $\alpha+\beta=2$, according to Lemma \ref{q1}, we have 
\[\psi(pq)=6= 2\Big(\binom{1}{0}\binom{1+0}{0}+\binom{1}{1}\binom{1+1}{1}\Big).\] 

Let $\psi(p^r q^s)=2^r \sum_{i=0}^s \binom{s}{i}\binom{r+i}{i}$, for each $r$ and $s$ with $r\geqslant s$ and $r+s\leqslant\alpha+\beta$. For two arbitrary positive integers $k$ and $\ell$ with $k\geqslant\ell$ and $k+\ell=\alpha+\beta+1$, we have
\begin{align*}
\psi(p^kq^\ell)&=2\psi(p^{k-1} q^{\ell})+\sum_{d\Vert q^{\ell}}\psi(p^{k}d)\\
&=2\psi(p^{k-1} q^{\ell})+\sum_{j=0}^{\ell-1}\psi(p^k q^j) \\
&=2\psi(p^{k-1} q^{\ell}) +\sum_{j=0}^{\ell-1}2^k\sum_{i=0}^j\binom{j}{i}\binom{k+i}{i}\\
&=2\psi(p^{k-1} q^{\ell}) +2^k\sum_{i=0}^{\ell-1}\binom{k+i}{i}\sum_{j=i}^{\ell-1}\binom{j}{i}\\
&=2^k\sum_{i=0}^\ell\binom{\ell}{i}\binom{k-1+i}{i} +2^k\sum_{i=0}^{\ell-1}\binom{k+i}{i}\binom{\ell}{i+1}\\
&=2^k\sum_{i=0}^\ell\binom{\ell}{i}\binom{k-1+i}{i} +2^k\sum_{i=1}^{\ell}\binom{\ell}{i}\binom{k+i-1}{i-1}\\
&=2^k\sum_{i=1}^\ell\binom{\ell}{i}\Big[\binom{k-1+i}{i} +\binom{k+i-1}{i-1}\Big]+2^k\\
&=2^k\sum_{i=0}^\ell\binom{\ell}{i}\binom{k+i}{i}.
\end{align*}
Note that the above computations is valid by the induction hypothesis, since $k-1+\ell=\alpha+\beta$ and $k+j\leqslant k+\ell-1\leqslant \alpha+\beta$. Moreover, note that the above arguments are correct also in the case $k=\ell$.
\end{proof}

\begin{theorem}\label{pqsecond}
Let $p$ and $q$ be two distinct prime numbers and let $\alpha\geqslant \beta$ be two positive integers. Suppose also that $f(x)=x^\alpha(x+1)^\beta$. Then $\psi(p^\alpha q^\beta)=\frac{2^\alpha f^{(\alpha)}(1)}{\alpha!}$.
\end{theorem}

\begin{proof}
Let $g_i(x)=x^{\alpha+i}$. Thus
\[f(x)=x^\alpha\sum_{i=0}^\beta\binom{\beta}{i} x^i =\sum_{i=0}^\beta \binom{\beta}{i}g_i(x).\]
It is straightforward to see that $\frac{g^{(\alpha)}(x)}{\alpha!}=\binom{\alpha+i}{i}x^i$. Thus we have
\[
\frac{f^{(\alpha)}(x)}{\alpha!}=\sum_{i=0}^\beta \binom{\beta}{i}\frac{g^{(\alpha)}(x)}{\alpha!}
=\sum_{i=0}^\beta \binom{\beta}{i}\binom{\alpha+i}{i}x^i.
\]
Now, according to Theorem \ref{pqfirst}, we have the result. 
\end{proof}

\begin{example}
We calculate $\psi(p^\alpha q^3)$ for any positive integer $\alpha$ greater than or equal to 3. Using Theorem \ref{pqfirst}, we can express it as:
\begin{align*}
\psi(p^\alpha q^3)&=2^\alpha\sum_{i=0}^3\binom{3}{i} \binom{\alpha+i}{i}\\
&=\binom{3}{0}\binom{\alpha+0}{0}
+\binom{3}{1}\binom{\alpha+1}{1}
+\binom{3}{2}\binom{\alpha+2}{2}
+\binom{3}{3}\binom{\alpha+3}{3}\\
&= 1+3(\alpha+1)+3\cdot \frac{(\alpha+2)(\alpha+1)}{2}+\frac{(\alpha+3)(\alpha+2)(\alpha+1)}{6}\\
&=\frac{1}{6}(\alpha^3+12\alpha^2+38\alpha+36).
\end{align*}
As an alternative method, we can represent $f(x)=x^\alpha(x+1)^3=x^{\alpha+3}+3x^{\alpha+2} +3x^{\alpha+1}+x^\alpha$ and employ Theorem \ref{pqsecond} to calculate:
\begin{align*}
\frac{f^{(\alpha)}(1)}{3!}&=\binom{\alpha+3}{3}1^3+3\binom{\alpha+2}{2}1^2 +3\binom{\alpha+1}{1}1^1+\binom{\alpha+0}{0}1^0.
\end{align*}

\end{example}

\bibliographystyle{amsplain}

\begin{thebibliography}{99}

\bibitem{BermanKohler1976} J. Berman, and P. K\"{o}hler, (1976). \emph{Cardinalities of finite distributive lattices.} Mitt. Math. Sem. Giessen, 121, 103–124. 

\bibitem{Church1940} R. Church, (1940). \emph{Numerical analysis of certain free distributive structures.} Duke Mathematical Journal, 6(3), 732–734. 

\bibitem{Dedekind1897} R. Dedekind, (1897). \emph{\"{U}ber Zerlegungen von Zahlen durch ihre gr\"{o}\ss ten gemeinsamen Teiler.} Gesammelte Werke, Vol. 2, pp. 103–148.

\bibitem{Kisielewicz1988} A. Kisielewicz, (1988). \emph{A solution of Dedekind's problem on the number of isotone Boolean functions.} Journal f\"{u}r die Reine und Angewandte Mathematik, 1988(386), 139–144. 

\bibitem{KleitmanMarkowsky1975} D. Kleitman, and G. Markowsky, (1975). \emph{On Dedekind's problem: the number of isotone Boolean functions. II.} Transactions of the American Mathematical Society, 213, 373–390. 

\bibitem{Korshunov1981} A. D. Korshunov, (1981). \emph{The number of monotone Boolean functions.} Problemy Kibernet., 38, 5–108. MR 0640855.

\bibitem{Wiedemann1991} D. Wiedemann, (1991). \emph{A computation of the eighth Dedekind number.} Order, 8(1), 5–6. 

\bibitem{Zaguia1993} N. Zaguia, (1993). \emph{Isotone maps: enumeration and structure.} In N. W. Sauer, R. E. Woodrow, and B. Sands (Eds.), Finite and Infinite Combinatorics in Sets and Logic (Proc. NATO Advanced Study Inst., Banff, Alberta, Canada, May 4, 1991), pp. 421–430. Kluwer Academic Publishers. 

\end{thebibliography}

\section*{Statements and Declarations}

\textbf{Funding}
The authors declare that no funds, grants, or other support were received during the preparation of this manuscript.
\end{document}